\documentclass[11pt]{article}
\usepackage{graphicx}
\usepackage{amsmath,amsthm,amssymb}
\usepackage{xypic}
\usepackage{mathabx}
\usepackage{bbm}
\usepackage{color}
\newtheorem{theorem}{Theorem}[section]
\newtheorem*{theorem*}{Theorem}
\newtheorem{proposition}[theorem]{Proposition}
\newtheorem*{proposition*}{Proposition}
\newtheorem{definition}[theorem]{Definition}
\newtheorem*{definition*}{Definition}
\newtheorem{example}[theorem]{Example}

\newtheorem{remark}[theorem]{Remark}
\newtheorem{corollary}[theorem]{Corollary}
\newtheorem*{corollary*}{Corollary}
\newtheorem{lemma}[theorem]{Lemma}
\newtheorem{question}[theorem]{Question}

\newtheorem{notation}[theorem]{Notation}
\newtheorem{warning}[theorem]{Warning}

\newcommand\blfootnote[1]{
  \begingroup
  \renewcommand\thefootnote{}\footnote{#1}
  \addtocounter{footnote}{-1}
  \endgroup
}

\begin{document}


\title{Ramification and descent in homotopy theory and derived algebraic geometry}
\author{John D. Berman}
\maketitle

\begin{abstract}\blfootnote{The author was supported by an NSF Postdoctoral Fellowship under grant 1803089.}
We introduce notions of unramified and totally ramified maps in great generality -- for commutative rings, schemes, ring spectra, or derived schemes. We prove that the definition is equivalent to the classical definition in the case of rings of integers in number fields.

The new definition leads directly (without computational techniques) to a calculation of topological Hochschild homology for rings of integers. We show that $\text{THH}(R)$ is the homotopy cofiber of a map $R[\Omega S^3]\otimes_R\Omega^1_{R/\mathbb{Z}}\to R[\Omega S^3\langle 3\rangle]$, so there is a long exact sequence $\cdots\to H_\ast(\Omega S^3;\Omega^1_{R/\mathbb{Z}})\to H_\ast(\Omega S^3\langle 3\rangle;R)\to\text{THH}_\ast(R)\to\cdots$.

Any time an extension $Y/X$ is a composite of unramified and totally ramified extensions, our results allow for the study of $\text{THH}(X)$ in terms of $\text{THH}(Y)$ by a kind of weak \'etale descent (\emph{ramified descent}).
\end{abstract}


\section{Introduction}
\subsection{Ramification}
\noindent A central theme in algebraic number theory is the study of ramification in number fields. If $K\to L$ is an extension of number fields, with associated rings of integers $R\to A$, and if $\mathfrak{p}$ is a prime ideal in $R$, the ramification describes how $\mathfrak{p}$ factors into prime ideals in $A$. The extension $A/R$ is \emph{unramified} at $\mathfrak{p}$ if this factorization is squarefree, or \emph{totally ramified} if $\mathfrak{p}$ is a power of a single prime ideal.

The language comes from geometry: We can regard $\text{Spec}(A)\to\text{Spec}(R)$ as being nearly a covering space map, with some bad behavior at points (primes) which are ramified.

We will generalize these definitions to (algebraically) any extension of rings or even ring spectra or (geometrically) any extension of schemes, derived schemes, or even spectral schemes. We will first give the geometric definition for intuition, although it is the algebraic definition which we will use throughout the paper.

Take arbitrary spectral schemes $\ast\to Y\to X$ (no definition necessary for this paper, but see \cite{HAG2} or \cite{SAG}), and think of $\ast$ as a point in $Y$. Call the homotopy pullback $\ast\times^h_Y\ast$ the space $\Omega_\ast Y$ of based loops and similarly $\Omega_\ast X=\ast\times^h_X\ast$. There are induced maps $\Omega_\ast Y\xrightarrow{f}\Omega_\ast X$ and $\ast\xrightarrow{i}\Omega_\ast X$, where the image of $i$ is thought of as the constant loops in $X$.

\begin{definition}
With $\ast\to Y\to X$ as above, let $\mathcal{F}$ be the homotopy fiber of $\mathcal{O}(\Omega_\ast X)\to f_\ast\mathcal{O}(\Omega_\ast Y)$, or the sheaf of functions on $\Omega_\ast X$ which vanish on loops that lift to $Y$. We say that $Y/X$ is:
\begin{itemize}
\item \emph{unramified at $y$} if $i^\ast\mathcal{F}\cong 0$ (read: $\mathcal{F}$ is trivial at constant loops);
\item \emph{totally ramified at $y$} if $\mathcal{F}\cong i_\ast\mathcal{F}_0$ for some sheaf $\mathcal{F}_0$ defined on $\ast$ (read: $\mathcal{F}$ is trivial away from constant loops).
\end{itemize}
\end{definition}

\noindent Let us unpack what this definition actually means algebraically, in the case where $X=\text{Spec}(R)$, $Y=\text{Spec}(A)$, $\ast=\text{Spec}(k)$.

\begin{definition}\label{DefTotalRam}
If $R\to A\to k$ are commutative\footnote{For us, a commutative ring spectrum means an $\mathbb{E}_\infty$-ring spectrum.} ring spectra, let $I_{A/R}^k$ be the homotopy fiber of the map $k\otimes_Rk\to k\otimes_Ak$ (which is a $k\otimes_Rk$-module). We say the extension $A/R$ is:
\begin{itemize}
\item \emph{unramified} at $k$ if $I_{A/R}^k\otimes_{k\otimes_Rk}k\cong 0$.
\item \emph{totally ramified} at $k$ if there is a $k$-module structure on $I_{A/R}^k$ which induces the intrinsic $k\otimes_Rk$-module structure by restriction along the multiplication map $k\otimes_Rk\to k$.
\end{itemize}
\end{definition}

\begin{remark}
These definitions work perfectly well for ordinary commutative rings and ordinary schemes. Any commutative ring $A$ is an example of a commutative ring spectrum (the discrete or Eilenberg-Maclane spectrum) which we also denote $A$. The tensor products are derived, and the homotopy fiber $k\otimes_Rk\to k\otimes_Ak$ can be computed in the sense of homological algebra.

In particular, every time we write a tensor product of ordinary rings or modules, we take it to be derived.
\end{remark}

\noindent Our first main result justifies the terms \emph{unramified} and \emph{totally ramified}.

\begin{theorem*}[\ref{ThmEquiv}]
Suppose $L/K$ is an extension of number fields, $A/R$ are the rings of integers, $\mathfrak{p}$ is a prime in $R$, and $\mathfrak{p}^\prime|\mathfrak{p}$ is a prime in $A$. Then:
\begin{itemize}
\item $A/R$ is totally ramified at $\mathfrak{p}$ (in the classical sense) if and only if $A/R$ is totally ramified at $A/\mathfrak{p}^\prime$ (in the sense of Definition \ref{DefTotalRam});
\item $A/R$ is unramified at $\mathfrak{p}$ (in the classical sense) if and only if $A/R$ is unramified at $A/\mathfrak{p}^\prime$ (in the sense of Definition \ref{DefTotalRam}).
\item (Proposition \ref{PropFF}) $A/R$ is unramified at $L$
\end{itemize}
\end{theorem*}

\noindent Note that this theorem describes the ramification of $\text{Spec}(A)/\text{Spec}(R)$ at \emph{every} point, even the generic point (where it is always unramified).

The deepest part of this theorem is as follows. If $A/R$ is totally ramified at $\mathfrak{p}$ with residue field $k=A/\mathfrak{p}^\prime$, we are asserting that $I_{A/R}^k$ has a $k$-module structure. In fact, we identify this $k$-module explicitly. Let $\Omega^1_{A/R}$ be the $A$-module of K\"ahler differentials.

\begin{theorem*}[\ref{MainLemma}]
If $A/R$ is totally ramified at $\mathfrak{p}$ with residue field $k$, then $I_{A/R}^k\cong\Omega^1_{A/R}\otimes_Ak$ as $k\otimes_Rk$-modules. That is, there is a fiber sequence $$\Omega^1\otimes_Ak\to k\otimes_Rk\to k\otimes_Ak.$$
\end{theorem*}

\subsection{THH of rings of integers}
\noindent The ramification of an extension of commutative ring spectra $A/R$ is closely connected to topological Hochschild homology, which is the commutative ring spectrum $\text{THH}^R(A)=A\otimes_{A\otimes_RA}A$.

If $R$ and $A$ are ordinary rings, then (as usual) all tensor products are derived, and $\text{THH}^R(A)$ is a differential graded algebra. Its homology groups are the classical Hochschild homology $\text{HH}_\ast^R(A)=\text{THH}_\ast^R(A)$.

The relationship between THH and ramification is suggested by the close relationship between THH and loop spaces:

\begin{remark}
Geometrically, if $X=\text{Spec}(R)$ and $Y=\text{Spec}(A)$, then $\text{Spec}(\text{THH}^R(A))=Y\times^h_{Y\times^h_XY}Y=LY$, the \emph{free loop space} of $Y$ in $\text{Sch}_{/X}$.
\end{remark}

\begin{remark}
Many authors have used THH to study ramification (as an informal concept) of extensions arising in stable homotopy theory. These include Blumberg-Mandell \cite{BM}, who confirmed a philosophy of Hesselholt that $ku$ (connective complex $k$-theory) is a tamely ramified extension of its Adams summand $\ell$, Dundas-Lindenstrauss-Richter \cite{DLR} and H\"oning-Richter \cite{HR} who studied $ku/ko$ and various extensions related to $tmf$, and others.
\end{remark}

\noindent Our second main result uses Theorem \ref{ThmEquiv} to identify the spectrum $\text{THH}(R)$, where $R$ is a ring of integers in a number field.

\begin{theorem*}[\ref{MainThm}]
If $R$ is a ring of integers in a number field, $p$ is a prime, and $S$ is a commutative ring spectrum with a map to $\mathbb{Z}_p$, there is a homotopy fiber/cofiber sequence (only defined after $p$-completion) $$\text{THH}^{S[z]}(R)\otimes_R\Omega^1_{R/\mathbb{Z}}\to\text{THH}^S(\mathbb{Z})\otimes_\mathbb{Z}R\to\text{THH}^S(R)$$ of $\text{THH}^S(\mathbb{Z})\otimes_\mathbb{Z}R$-module spectra.
\end{theorem*}

\noindent Let us make clear what we mean by $\text{THH}^{S[z]}(R)$. If $S$ is a ring spectrum, we write $S[z]$ for the monoid ring spectrum $S[\mathbb{N}]=S\otimes\Sigma^\infty_{+}\mathbb{N}$.

The $p$-completion of a ring of integers is a DVR (discrete valuation ring), and we always regard $R_p$ as an $\mathbb{S}[z]$-algebra by mapping $z$ to a uniformizer. Then the theorem specifically asserts the existence of a fiber sequence $$\text{THH}^{S[z]}(R_p)\otimes_R\Omega^1_{R/\mathbb{Z}}\to\text{THH}^S(\mathbb{Z}_p)\otimes_{\mathbb{Z}_p}R_p\to\text{THH}^S(R_p)\cong\text{THH}^S(R)^\wedge_p.$$ An important special case is when $S=\mathbb{S}$ is the sphere spectrum. As the initial commutative ring spectrum, $\mathbb{S}$ admits a unique map to any other commutative ring spectrum. When $S=\mathbb{S}$, it is common practice to omit it from the notation, writing $\text{THH}(R)=\text{THH}^\mathbb{S}(R)$.

For formal reasons related to Thom spectra \cite{Blumberg}, $\text{THH}(\mathbb{Z}_p)$ is equivalent to the $\mathbb{E}_1$-group ring spectrum $\mathbb{Z}_p[\Omega S^3\langle 3\rangle]$, where $S^3\langle 3\rangle$ is the 3-connective cover of $S^3$. For similar reasons \cite{KN}, if $R$ is a ring of integers, then $\text{THH}^{\mathbb{S}[z]}(R_p)\cong R_p[\Omega S^3]$ as $\mathbb{E}_1$-ring spectra. Therefore, Theorem \ref{MainThm} describes $\text{THH}(R)^\wedge_p$ as a cofiber of explicitly described $R$-module spectra $$\text{THH}(R)^\wedge_p\cong\text{cofib}(R_p[\Omega S^3]\otimes_R\Omega^1_{R/\mathbb{Z}}\to R_p[\Omega S^3\langle 3\rangle]).$$ This is the sense in which Theorem \ref{MainThm} is a calculation of $\text{THH}(R)$.

\begin{warning}
Theorem \ref{MainThm} actually makes the stronger assertion that this cofiber sequence is of $R_p[\Omega S^3\langle 3\rangle]$-modules. However, we do not know whether the $R_p[\Omega S^3\langle 3\rangle]$-module structure on $R_p[\Omega S^3]\otimes_R\Omega^1_{R/\mathbb{Z}}$ agrees with the canonical one given by the covering map $\Omega S^3\langle 3\rangle\to\Omega S^3$.
\end{warning}

\begin{remark}
The homotopy groups $\text{THH}_\ast(R)=\pi_\ast\text{THH}(R)$ are already known. They were originally calculated by Lindenstrauss-Madsen \cite{LM} using computational techniques. A simpler calculation by Krause-Nikolaus \cite{KN} uses $\text{THH}_\ast^{\mathbb{S}[z]}(R_p)\cong R_p[x]$ as we do, but is otherwise unrelated to our methods.

In contrast, our proof of Theorem \ref{MainThm} is not calculational, and similar arguments can be used for any extension which is a composite of unramified and totally ramified extensions. The calculations take place in the proof of Theorem \ref{ThmEquiv} (Section 2), which involves algebraic number theory but no homotopy theory.
\end{remark}

\noindent Let us say a word about the homotopy groups $\text{THH}_\ast(R)$. If $K$ is the number field associated to $R$, the inverse different is the fractional ideal $$D^{-1}=\{x\in K|\text{tr}(xy)\in\mathbb{Z},\text{ }\forall y\in R\}.$$ Notice $R\subseteq D^{-1}$. Equivalently, $D^{-1}\cong\text{Hom}_\mathbb{Z}(R,\mathbb{Z})$, and $R\subseteq\text{Hom}_\mathbb{Z}(R,\mathbb{Z})$ via the trace. Lindenstrauss-Madsen \cite{LM} proved $$\text{THH}_\ast(R)\cong\begin{cases}R\text{ if }\ast=0, \\ D^{-1}/nR\text{ if }\ast=2n-1>0, \\ 0\text{ otherwise.}\end{cases}$$ This is related to our results as follows: The homotopy fiber sequence $R[\Omega S^3]\otimes_R\Omega^1_{R/\mathbb{Z}}\to R[\Omega S^3\langle 3\rangle]\to\text{THH}(R)$ induces a long exact sequence in $\pi_\ast$ (after $p$-completion), the first few terms of which are written out below. $$\xymatrix{
\cdots\ar[r] &H_\ast(\Omega S^3;\Omega^1_{R/\mathbb{Z}})\ar[r] &H_\ast(\Omega S^3\langle 3\rangle;R)\ar[r] &\text{THH}_\ast(R)\ar[r] &\cdots \\
H_0 &D^{-1}/R\ar[r] &R\ar[r] &R & \\
H_1 &0\ar[r] &R/1R\ar[r] &\text{THH}_1(R)\ar[llu] & \\
H_2 &D^{-1}/R\ar[r] &0\ar[r] &0\ar[llu] & \\
H_3 &0\ar[r] &R/2R\ar[r] &\text{THH}_2(R).\ar[llu] &
}$$ We are using the well-known formulas (which can be obtained from the Serre spectral sequence) $H_\ast(\Omega S^3;\mathbb{Z})\cong\mathbb{Z}[x]$, where $x$ is a degree 2 generator, and $$H_\ast(\Omega S^3\langle 3\rangle;\mathbb{Z})\cong\begin{cases}\mathbb{Z}\text{ if }\ast=0, \\ \mathbb{Z}/n\text{ if }\ast=2n-1>0, \\ 0\text{ otherwise.}\end{cases}$$ Therefore, we have short exact sequences $$0\to R/nR\to\text{THH}_{2n-1}(R)\to D^{-1}/R\to 0,$$ identifying $\text{THH}_{2n-1}(R)\cong D^{-1}/nR$.

\subsection{Ramified descent}
\noindent In another sense, Theorem \ref{MainThm} can be regarded as a generalization of the \'etale descent theorem of Geller-Weibel \cite{WG}, McCarthy-Minasian \cite{MM}, and Mathew \cite{Akhil} (in increasing order of generality):

\begin{theorem*}
If $S\to R\to A$ are commutative ring spectra and $A/R$ is \'etale (in particular, unramified), then the following map is an equivalence: $$\text{THH}^S(R)\otimes_RA\to\text{THH}^S(A).$$
\end{theorem*}

\noindent If $R\to A$ is not necessarily \'etale but has well-behaved ramification, we might ask how close this map is to an equivalence.

We introduce notation $\text{Ram}^S(A/R)$ for the homotopy fiber of the map $\text{THH}^S(R)\otimes_RA\to\text{THH}^S(A)$. Then \'etale descent states $\text{Ram}^S(A/R)\cong 0$ if $A/R$ is \'etale, while Theorem \ref{MainThm} asserts $$\text{Ram}^S(R/\mathbb{Z})^\wedge_p\cong\text{THH}^{S[z]}(R_p)\otimes_R\Omega^1_{R/\mathbb{Z}}$$ if $R$ is a ring of integers. We express the close relationship between ramification and THH in two results (proved just by unpacking definitions):

\begin{proposition*}[\ref{PropEasy}]
$A/R$ is unramified at $k$ if and only if $$\text{Ram}^R(A/R)\otimes_Ak\cong 0.$$
\end{proposition*}

\begin{proposition*}[\ref{PropKey}]
If $A/R$ is totally ramified at $k$, then $$\text{Ram}^S(A/R)\otimes_Ak\cong\text{THH}^S(k)\otimes_kI^k_{A/R}$$ as $\text{THH}^S(R)\otimes_Rk$-modules.
\end{proposition*}

\begin{remark}
Geometrically, we think of $X=\text{Spec}(R)$ and $Y=\text{Spec}(A)$. Consider the map $f:LY\to LX\times_XY$. The lifting problem for $f$ asks, given a loop in $X$ and a lift of its basepoint to $Y$, whether the loop lifts to $Y$. Then $\text{Ram}^S(A/R)$ corresponds to the sheaf $\mathcal{R}$ of functions defined on those loops in $LX\times_XY$ which do not lift to $LY$; that is, the fiber of $\mathcal{O}(LX\times_XY)\to f_\ast\mathcal{O}(LY)$.

The last two propositions state that $\mathcal{R}$ can be directly calculated anywhere that $Y/X$ is either unramified or totally ramified.
\end{remark}

\noindent In the event that $A/R$ factors as a composite of unramified and totally ramified extensions, these results should in principle allow us to calculate $\text{Ram}^S(A/R)$ itself via a local-to-global argument or \emph{ramified descent}. The key lemma that makes this possible is:

\begin{proposition*}[\ref{PropTriple}]
If $A\to B\to C\to k$ are commutative ring spectra, then the following are fiber sequences: $$I_{B/A}^k\to I_{C/A}^k\to I_{C/B}^k,$$ $$\text{Ram}^S(B/A)\otimes_BC\to\text{Ram}^S(C/A)\to\text{Ram}^S(C/B).$$
\end{proposition*}

\noindent This is precisely the strategy used to prove Theorem \ref{MainThm}. Indeed, any extension of number fields factors (at each prime) as an unramified extension followed by a totally ramified extension. Therefore, we can use this strategy to calculate $\text{Ram}(R/\mathbb{Z})$ when $R$ is a ring of integers in a number field.

\subsection{Acknowledgments, previous work, and questions}
\noindent This paper benefited from conversations with Owen Gwilliam, Jeremy Hahn, Achim Krause, Mike Mandell, and others. Special thanks are due to Rok Gregoric for detailed feedback and geometric intuition, and Andrew Blumberg for his continued support.

\begin{remark}\label{RmkGeneral}
We are certainly not the first to study the spectra $\text{Ram}^S(A/R)$. In the case $S=R$, $\text{Ram}^R(A/R)$ is just the fiber of $A\to\text{THH}^R(A)$, which is not far from $\text{THH}^R(A)$ itself. (It is the desuspension of reduced THH.) Dundas-Lindenstrauss-Richter \cite{DLR} and others have studied reduced THH as a measure of ramification of $A/R$.

Our $\text{Ram}^S(A/R)$ is a more refined measure of ramification, in the sense that $\text{Ram}^R(A/R)$ can be reconstructed from it. That is, for any map $S\to S^\prime$ of commutative ring spectra over $R$, $$\text{Ram}^{S^\prime}(A/R)\cong\text{Ram}^S(A/R)\otimes_{\text{THH}^S(S^\prime)}S^\prime.$$ (Corollary \ref{CorOuterBase}) Therefore, the most refined of all these invariants is $\text{Ram}^\mathbb{S}(A/R)$ and the coarsest is $\text{Ram}^R(A/R)$, or reduced THH.
\end{remark}

\noindent Weibel-Geller \cite{WG} were the first to prove \'etale descent $\text{THH}^\mathbb{Z}(R)\otimes_R A\cong\text{THH}^\mathbb{Z}(A)$ when $R$ and $A$ are ordinary (discrete) rings; they do so precisely by showing that $\text{Ram}^\mathbb{Z}(R/\mathbb{Z})$ vanishes at each prime. Our calculation of $\text{Ram}^\mathbb{S}(R/\mathbb{Z})$ when $R$ is a ring of integers generalizes their ideas.

In this paper, we mostly consider ramification for rings of integers in number fields, but there are a few results that are more general. Notably:

\begin{proposition*}[\ref{PropExample}]
If an extension $A/R$ is totally ramified everywhere (that is, at $k=A$), then it is also unramified everywhere, and $A\otimes_RA\to A$ is an equivalence.
\end{proposition*}

\noindent A number field is ramified at only finitely many primes. The last proposition can be regarded as a vast generalization of this statement -- it is impossible for any extension $A/R$ to be ramified \emph{everywhere} except in the degenerate case that it has Galois group $0$ (that is, $A\otimes_RA\cong A$).

Aside from what we consider in this paper, there are many interesting questions about ramification for spectra that arise in chromatic homotopy theory. If $A/R$ is a Galois extension of commutative ring spectra in the sense of Rognes \cite{Rognes}, the extension of connective covers is generally not a Galois extension, \cite{Akhil2} Theorem 6.17. H\"oning and Richter \cite{HR} propose studying the ramification of these extensions. For example:

\begin{question}
Where are the extensions $ku/ko$ and $ku/\ell$ unramified or totally ramified? Can ramified descent be used to relate $\text{THH}(ku)$ to $\text{THH}(ko)$ and $\text{THH}(\ell)$? What about analogous questions for $tmf$?
\end{question}

\noindent From another direction, we used ramification of $A/R$ to reduce the study of $\text{THH}(A)$ to $\text{THH}(R)$. If $A/R$ is \'etale (therefore unramified), \'etale descent methods allows us to study $A$ in terms of $R$ for a variety of homology theories: TP, TC, algebraic K-theory, etc.

\begin{question}
If $A/R$ is totally ramified, are there ramified descent formulas (analogous to \'etale descent) which relate $TP(R)$ and $TP(A)$, $TC(R)$ and $TC(A)$, $K(R)$ and $K(A)$, etc.?
\end{question}

\section{Number theory}
\begin{notation}
If $R$ is a commutative ring and $M$ an $R$-module, we also write $R$ for the Eilenberg-Maclane ring spectrum, $M$ for the Eilenberg-Maclane $R$-module spectrum. When we write $A_0\otimes_RA_1$, we mean the tensor product of spectra, or the derived tensor product $\pi_\ast(A_0\otimes_RA_1)\cong\text{Tor}_\ast^R(A_0,A_1)$. If instead we want to express the ordinary tensor product of discrete $R$-modules, we will use the notation $A_0\otimes^0_RA_1:=\text{Tor}_0^R(A_0,A_1)$.
\end{notation}

\noindent We will first summarize some basic number theory; see \cite{Neukirch} for reference.

Throughout this section, let $R\to A$ be a totally ramified extension of complete DVRs in mixed characteristic, with uniformizers $p\in R$ and $\pi\in A$. Then $A\cong R[\pi]/(f(\pi))$, the ideal $(p)=(\pi^e)\subseteq A$ for some $e\geq 1$ (the \emph{ramification index} of $p$), and $f(x)$ is an Eisenstein polynomial of degree $e$.

Since $A/R$ is flat, $A\otimes_RA\cong A\otimes_R^0A$. Let $I$ be the kernel of multiplication $A\otimes_RA\to A$, or equivalently the homotopy fiber since the multiplication map is surjective. So $I$ is an ideal of $A\otimes_RA$ and $(A\otimes_RA)/I\cong A$. The $A$-module of K\"ahler differentials is $$\Omega^1_{A/R}=I/I^2\cong I\otimes^0_{A\otimes_RA}(A\otimes_RA)/I\cong I\otimes^0_{A\otimes_RA}A.$$ If we identify $A\otimes_RA$ with $R[\pi_1,\pi_2]/(f(\pi_1),f(\pi_2))$, then $\Omega^1\cong A/(f^\prime(\pi))$ with generator $d\pi=\pi_1-\pi_2$ mod $I^2$.

Note $A/\pi\cong R/(f(0))$. Since $f(x)$ is an Eisenstein polynomial, $f(0)=pu$ where $u\in R$ is a unit. Hence, the induced map of residue fields $R/p\to A/\pi$ is an isomorphism. We denote by $k$ this residue field. Our goal in this section is to prove the following theorem, for which we need two lemmas:

\begin{theorem}\label{MainLemma}
With $A/R$ a totally ramified extension of complete DVRs in mixed characteristic (as above), the composite $$I\otimes_{A\otimes_RA}(k\otimes_Rk)\to\Omega^1\otimes_{A\otimes_RA}(k\otimes_Rk)\to\Omega^1\otimes_Ak$$ is a weak equivalence of chain complexes (or spectra).
\end{theorem}

\noindent The language of the theorem differs somewhat from how it was stated in the introduction, but the two are equivalent (Remark \ref{RmkRect} below).

\begin{lemma}\label{L1}
If $S\to R\to A$ are commutative ring spectra and $M$ is an $(R,R)$-bimodule over $S$ (that is, an $R\otimes_S R$-module), then the map of $(R,R)$-bimodules $M\to M\otimes_{R\otimes_S R}(A\otimes_S A)$ induces a map of $(A,A)$-bimodules $$\eta(M):A\otimes_RM\otimes_RA\to M\otimes_{R\otimes_S R}(A\otimes_S A)$$ which is an equivalence.
\end{lemma}

\begin{proof}
This $\eta$ is a natural transformation of functors $\text{Mod}_{R\otimes_S R}\to\text{Mod}_{A\otimes_S A}$, each of which preserves colimits, and $\eta(R\otimes_S R)$ is an isomorphism. Therefore, $\eta(M)$ is a natural isomorphism for any $M$.
\end{proof}

\begin{remark}\label{RmkBlah}
$I$ is the homotopy fiber of $A\otimes_RA\to A$, so $I\otimes_{A\otimes_RA}(k\otimes_Rk)$ is the homotopy fiber of $k\otimes_Rk\to A\otimes_{A\otimes_RA}(k\otimes_Rk)$. Applying Lemma \ref{L1}, $I\otimes_{A\otimes_RA}(k\otimes_Rk)\cong k\otimes_AI\otimes_Ak$ is the homotopy fiber of $k\otimes_Rk\to k\otimes_Ak$.

In other words, Theorem \ref{MainLemma} asserts that there is a fiber sequence $$\Omega^1\otimes_Ak\to k\otimes_Rk\to k\otimes_Ak.$$ (We prefer the notation of Theorem \ref{MainLemma}, since the map $\Omega^1\otimes_Ak\to k\otimes_Rk$ is not easily described.)
\end{remark}

\begin{remark}\label{RmkRect}
If $A/R$ is an extension of rings of integers in number fields, totally ramified at $\mathfrak{p}$ with residue field $k$, then the completion $A_p/R_p$ is a totally ramified extension of complete DVRs in mixed characteristic. By Theorem \ref{MainLemma}, there is a fiber sequence $\Omega^1_{A_p/R_p}\otimes_{A_p}k\to k\otimes_{R_p}k\to k\otimes_{A_p}k$. Since $\Omega^1_{A_p/R_p}\cong(\Omega^1_{A/R})^\wedge_p$ and $k$ is already $p$-complete, it follows that there is a fiber sequence $\Omega^1_{A/R}\otimes_Ak\to k\otimes_Rk\to k\otimes_Ak$, which is how the theorem is stated in the introduction.
\end{remark}

\begin{lemma}\label{L2}
$H_\ast(\Omega^1\otimes_Ak)\cong H_\ast(k\otimes_AI\otimes_Ak)\cong\begin{cases}k\text{ if }\ast=0\text{ or }1 \\ 0\text{ otherwise.}\end{cases}$
\end{lemma}

\begin{proof}
Since $A$ is a DVR, $(f^\prime(\pi))=(\pi^d)$ as ideals for some $d$ (the \emph{differential exponent}). Therefore $\Omega^1=A/(\pi^d)\langle d\pi\rangle$. Resolving $k$ as $A\xrightarrow{\pi}A$, we can resolve $\Omega^1\otimes_Ak$ as the chain complex $A/(\pi^d)\langle d\pi\rangle\xrightarrow{\pi}A/(\pi^d)\langle d\pi\rangle$, which has $H_0\cong k$ generated by $d\pi$ and $H_1\cong k$ generated by $\pi^{d-1}d\pi$.

On the other hand, since $I$ is the homotopy fiber of $A\otimes_RA\to A$, then $I\otimes_Ak$ is the homotopy fiber of $A\otimes_Rk\to k$, or equivalently (since $k\cong R/p$) of $A/p\to A/\pi$. Since $f$ is an Eisenstein polynomial, $f(x)\equiv x^e$ mod $p$; therefore, $A/p\cong k[\pi]/(\pi^e)$, and the projection to $A/\pi\cong k$ sends $\pi\mapsto 0$.

Hence, $I\otimes_Ak$ is equivalent to the fiber of $k[\pi]/(\pi^e)\to k$, $\pi\mapsto 0$. This is surjective, so the fiber is the kernel, $I\otimes_Ak\cong\pi k[\pi]/(\pi^e)$. Therefore, $k\otimes_A(I\otimes_Ak)$ is equivalent to the chain complex $\pi k[\pi]/(\pi^e)\xrightarrow{\pi}\pi k[\pi]/(\pi^e)$, which has homology $k$ in degree $0$ generated by $\pi$ and $k$ in degree $1$ generated by $\pi^{e-1}$.
\end{proof}

\begin{proof}[Proof of Theorem \ref{MainLemma}]
We will prove $\phi:k\otimes_AI\otimes_Ak\to\Omega^1\otimes_Ak$ is a weak equivalence. Recall that $I$ is an ideal in $A\otimes_RA=R[\pi_1,\pi_2]/(f(\pi_1),f(\pi_2))$.

Resolving $k$ as the chain complex $A\xrightarrow{\pi}A$, $k\otimes_AI\otimes_Ak$ admits a weak equivalence from $$I\xrightarrow{\alpha}I\oplus I\xrightarrow{\beta}I$$ given by $\alpha(i)=(\pi_2i,\pi_1i)$ and $\beta(i_1,i_2)=\pi_1i_1-\pi_2i_2$. Unpacking, the map to $\Omega^1\otimes_Ak$ is given by $$\xymatrix{
I\ar[rr]^{(\pi_2i,\pi_1i)}\ar[d] &&I\oplus I\ar[d]_{i_1-i_2}\ar[rr]^{\pi_1i_1-\pi_2i_2} &&I\ar[d]^i &I\otimes_{A\otimes_RA}(k\otimes_Rk)\ar[d]^{\phi}\\
0\ar[rr] &&\Omega^1\ar[rr]_\pi &&\Omega^1 &\Omega^1\otimes_Ak.
}$$ We want to show that $\phi$ is an isomorphism in homology. By Lemma \ref{L2}, $H_2(\phi):0\to 0$, so it is an isomorphism. On the other hand, $H_0(\phi)$ is surjective. Since it is a map of $k$-modules $k\to k$ (Lemma \ref{L2}), it is an isomorphism.

It remains to prove that $H_1(\phi)$ is an isomorphism, which is also a map of $k$-modules $k\to k$ (Lemma \ref{L2}). It will suffice to construct an element $(\omega_1,\omega_2)\in I\oplus I$ such that $\pi_1\omega_1-\pi_2\omega_2=0$ (so it describes a class in $H_1$) and $\omega_1-\omega_2\neq 0$ (so $H_1(\phi)$ is not identically $0$). To do this, fix notation $$f(x)=x^e-pxg(x)-up,$$ where $u$ is a unit in $R$. (This is the general form of an Eisenstein polynomial.) For any polynomial $h(x)$, write $dh(\pi)=h(\pi_1)-h(\pi_2)\in I$. Set $$(\omega_1,\omega_2)=(d\pi^{e-1},u^{-1}g(\pi_1)(\pi_2^{e-1}-pg(\pi_2))d\pi+p\,dg(\pi)-\pi_2^{e-2}d\pi)\in I\oplus I.$$ That is, $\omega_1=\pi_1^{e-1}-\pi_2^{e-2}$ and (expanding and using $\pi_2^e-p\pi_2g(\pi_2)=up$), $$\omega_2=u^{-1}\pi_1\pi_2^{e-1}g(\pi_1)-u^{-1}p\pi_1g(\pi_1)g(\pi_2)-pg(\pi_2)-\pi_1\pi_2^{e-2}+\pi_2^{e-1}.$$ If we expand $\pi_1\omega_1-\pi_2\omega_2$ and use $\pi_1^e+p\pi_2g(\pi_2)-\pi_2^e=p\pi_1g(\pi_1)$, we see $$\pi_1\omega_1-\pi_2\omega_2=-u^{-1}\pi_1g(\pi_1)(\pi_2^e-p\pi_2g(\pi_2)-up)=0,$$ so $(\omega_1,\omega_2)$ is a class in $H_1$.

On the other hand, in $\Omega^1$ we can use identities $dh(\pi)=h^\prime(\pi)d\pi$ and $\pi_1=\pi_2=\pi$ to calculate $$\omega_1-\omega_2=(e\pi^{e-2}-u^{-1}g(\pi)(\pi^{e-1}-pg(\pi))-pg^\prime(\pi))d\pi.$$ All that remains is to show that this class is nonzero in $\Omega^1$. Using the relation $\pi^e-p\pi g(\pi)=up$, we have (in $I$) $$\pi(\omega_1-\omega_2)=(e\pi^{e-1}-pg(\pi)-p\pi g^\prime(\pi))d\pi=f^\prime(\pi)d\pi.$$ Since the annihilator of $d\pi\in\Omega^1$ is $(f^\prime(\pi))\subseteq A$, then $\pi(\omega_1-\omega_2)=f^\prime(\pi)d\pi$ implies $\omega_1-\omega_2\neq 0$ in $\Omega^1$, completing the proof.
\end{proof}

\section{Topological Hochschild homology}
\noindent If $S\to R$ are commutative ring spectra and $M$ is an $R\otimes_SR$-module,  $$\text{THH}^S(R;M)=R\otimes_{R\otimes_SR}M.$$ We always regard $R$ as an $R\otimes_SR$-algebra by multiplication $R\otimes_SR\to R$.

If $S=\mathbb{S}$, we drop it from the notation, writing $\text{THH}(R;M)$. We also write $\text{THH}^S(R)$ for $\text{THH}^S(R;R)$.

\begin{proposition}[Inner base change]\label{PropInnerBase}
If $S\to R\to A$ are commutative ring spectra and $M$ is an $R\otimes_SR$-module, then $$\text{THH}^S(R;M)\otimes_RA\cong\text{THH}^S(A;A\otimes_RM\otimes_RA).$$
\end{proposition}

\begin{proof}
$\text{THH}^S(R;M)\otimes_RA=M\otimes_{R\otimes_SR}A\cong M\otimes_{R\otimes_SR}(A\otimes_SA)\otimes_{A\otimes_SA}A$, which is equivalent to $(A\otimes_RM\otimes_RA)\otimes_{A\otimes_SA}A=\text{THH}(A;A\otimes_RM\otimes_RA)$ by Lemma \ref{L1}.
\end{proof}

\noindent Therefore, multiplication $A\otimes_RA\to A$ induces a map $$\xymatrix{
\text{THH}^S(R)\otimes_RA\ar@{=}[d]\ar[r] &\text{THH}^S(A)\ar@{=}[d] \\
\text{THH}^S(A;A\otimes_RA)\ar[r] &\text{THH}^S(A;A).
}$$ Unpacking, this top map agrees with the evident maps $\text{THH}^S(R)\to\text{THH}^S(A)$ and $A\to\text{THH}^S(A)$.

\begin{definition}
If $S\to R\to A$ are commutative ring spectra, $\text{Ram}^S(A/R)$ is the homotopy fiber of $\text{THH}^S(R)\otimes_RA\to\text{THH}^S(A)$; equivalently, $$\text{Ram}^S(A/R)\cong\text{THH}^S(A;I_{A/R}).$$
\end{definition}

\noindent Here we write $I_{A/R}$ for the fiber of multiplication $A\otimes_RA\to A$. More generally, if $k$ is an $A$-algebra, we write $I_{A/R}^k$ for the fiber of $k\otimes_Rk\to k\otimes_Ak$, so that $I_{A/R}=I_{A/R}^A$. As in Remark \ref{RmkBlah}, $I_{A/R}^k\cong k\otimes_AI_{A/R}\otimes_Ak$. Hence:

\begin{remark}\label{RmkBlah2}
By inner base change, $\text{Ram}^S(A/R)\otimes_Ak\cong\text{THH}^S(k;I_{A/R}^k)$.
\end{remark}

\begin{example}[\'etale descent \cite{Akhil}]\label{ExED}
If $R\to A$ is \'etale, then $\text{Ram}^S(A/R)\cong 0$.
\end{example}

\begin{example}
Suppose that $R,A$ are discrete and $A$ is flat over $R$. Then $I_{A/R}$ is also discrete; it is simply the kernel of multiplication $A\otimes_RA\to A$. Since $A\cong(A\otimes_RA)/I$, we have $\text{Ram}_0^R(A/R)=A\otimes^0_{A\otimes_RA}I\cong I/I^2$, or $$\Omega^1_{A/R}=\text{Ram}_0^R(A/R).$$ Even if $A/R$ is not necessarily flat, then $\pi_0I_{A/R}$ is the kernel of $A\otimes_RA\to A$, and it is still true that $\Omega^1_{A/R}\cong\text{Ram}_0^R(A/R)$.
\end{example}

\section{Formulas for $I_{A/R}^k$ and $\text{Ram}^S(A/R)$}
\begin{proposition}\label{PropComplete}
If $R\to A\to k$ are commutative ring spectra and $p$ is a prime, then $(I_{A/R}^k)^\wedge_p\cong I_{A^\wedge_p/R^\wedge_p}^{k^\wedge_p}$ and $\text{Ram}^S(A/R)^\wedge_p\cong\text{Ram}^{S^\wedge_p}(A^\wedge_p/R^\wedge_p)$. If $R$ and $A$ are discrete, then $(\Omega^1_{A/R})^\wedge_p\cong\Omega^1_{A^\wedge_p/R^\wedge_p}$.
\end{proposition}

\begin{proof}
Since $p$-completion preserves tensor products and fiber sequences, $$(I_{A/R}^k)^\wedge_p\to k^\wedge_p\otimes_{R^\wedge_p}k^\wedge_p\to k^\wedge_p\otimes_{A^\wedge_p}k^\wedge_p$$ is a fiber sequence, which is to say $(I_{A/R}^k)^\wedge_p\cong I_{A^\wedge_p/R^\wedge_p}^{k^\wedge_p}$. Therefore, $$\text{Ram}^S(A/R)^\wedge_p\cong A^\wedge_p\otimes_{A^\wedge_p\otimes_{S^\wedge_p}A^\wedge_p}(I_{A/R})^\wedge_p\cong\text{Ram}^{S^\wedge_p}(A^\wedge_p/R^\wedge_p).$$ The last claim follows by $\Omega^1_{A/R}=\pi_0\text{Ram}^R(A/R)$.
\end{proof}

\begin{proposition}\label{PropTriple}
If $A\to B\to C\to k$ are commutative ring spectra, then the following are fiber sequences: $$I_{B/A}^k\to I_{C/A}^k\to I_{C/B}^k,$$ $$\text{Ram}^S(B/A)\otimes_BC\to\text{Ram}^S(C/A)\to\text{Ram}^S(C/B).$$
\end{proposition}

\begin{proof}
All the squares in the following diagrams are homotopy pullbacks. In fact, the second diagram is obtained from the first by choosing $k=C$ and applying the functor $\text{THH}^S(C;-)$.) $$\xymatrix{
I_{B/A}^k\ar[r]\ar[d] &I_{C/A}^k\ar[r]\ar[d] &k\otimes_Ak\ar[d] \\
0\ar[r] &I_{C/B}^k\ar[r]\ar[d] &k\otimes_Bk\ar[d] \\
&0\ar[r] &k\otimes_Ck,
}$$ $$\xymatrix{
\text{Ram}^S(B/A)\otimes_BC\ar[r]\ar[d] &\text{Ram}^S(C/A)\ar[r]\ar[d] &\text{THH}^S(A)\otimes_AC\ar[d] \\
0\ar[r] &\text{Ram}^S(C/B)\ar[r]\ar[d] &\text{THH}^S(B)\otimes_BC\ar[d] \\
&0\ar[r] &\text{THH}^S(C).
}$$
\end{proof}

\noindent In $\text{THH}^S(R)$, we can vary the argument $R$ using the inner base change formula (Proposition \ref{PropInnerBase}). There is also an outer base change formula used to vary $S$, which we describe in the remainder of this section. We won't use it in any proofs, but we cited it in the introduction (Remark \ref{RmkGeneral}) to justify that $\text{Ram}^S(A/R)$ is a strictly stronger invariant than reduced THH, and $\text{Ram}(A/R)=\text{Ram}^\mathbb{S}(A/R)$ is the strongest of all.

Fix commutative ring spectra $S\to S^\prime\to R$ and an $R\otimes_{S^\prime}R$-module $M$. We may regard $R$ and $M$ as modules over $R\otimes_{S^\prime}R\cong S^\prime\otimes_{S^\prime\otimes_SS^\prime}(R\otimes_SR)$ and $R\otimes_{S^\prime}R\cong(R\otimes_SR)\otimes_{S^\prime\otimes_SS^\prime}S^\prime$, respectively. It follows that

\begin{center}$\text{THH}^S(R;M)=R\otimes_{R\otimes_SR}M$ is a module over $S^\prime\otimes_SS^\prime=\text{THH}^S(S^\prime)$\end{center}

\noindent For the same reason, $\text{THH}^{S^\prime}(R;M)$ also has a $\text{THH}^S(S^\prime)$-module structure, but this one is trivial -- it is induced from the $S^\prime$-module structure by restriction along $\text{THH}^S(S^\prime)\to\text{THH}^{S^\prime}(S^\prime)=S^\prime$. Therefore, the map of $\text{THH}^S(S^\prime)$-modules $\text{THH}^S(R;M)\to\text{THH}^{S^\prime}(R;M)$ induces \begin{equation}\label{Eqn1}
\text{THH}^S(R;M)\otimes_{\text{THH}^S(S^\prime)}S^\prime\to\text{THH}^{S^\prime}(R;M).
\end{equation}

\begin{proposition}[Outer base change]\label{PropOuterBase}
For any $S\to S^\prime\to R$ and $R\otimes_{S^\prime}R$-module $M$, the map (\ref{Eqn1}) is an equivalence.
\end{proposition}

\begin{proof}
Each side in (\ref{Eqn1}) is functorial in $M$, and the map is a natural transformation $\eta_M$ of functors $\text{Mod}_{R\otimes_{S^\prime}R}\to\text{Mod}_R$, each of which preserves colimits. Therefore, it suffices to prove that $$\eta_{R\otimes_{S^\prime}R}:\text{THH}^S(R;R\otimes_{S^\prime}R)\otimes_{\text{THH}^S(S^\prime)}S^\prime\to\text{THH}^{S^\prime}(R;R\otimes_{S^\prime}R)\cong R$$ is an equivalence. Unpacking, this is a consequence of inner base change (Proposition \ref{PropInnerBase}) since $\text{THH}^S(R;R\otimes_{S^\prime}R)\cong R\otimes_{S^\prime}\text{THH}^S(S^\prime)$.
\end{proof}

\begin{corollary}\label{CorOuterBase}
If $S\to S^\prime\to R\to A$ are commutative ring spectra, $$\text{Ram}^S(A/R)\otimes_{\text{THH}^S(S^\prime)}S^\prime\cong\text{Ram}^{S^\prime}(A/R).$$
\end{corollary}

\section{Ramification for commutative ring spectra}
\noindent Recall Definition \ref{DefTotalRam} -- We say $A/R$ is:
\begin{itemize}
\item \emph{unramified} at $k$ if $I_{A/R}^k\otimes_{k\otimes_Rk}k\cong 0$;
\item \emph{totally ramified} at $k$ if $I_{A/R}^k$ is a $k$-module (with trivial $k\otimes_Rk$-module structure induced by multiplication $k\otimes_Rk\to k$).
\end{itemize}

\begin{proposition}\label{PropEasy}
$I_{A/R}^k\otimes_{k\otimes_Rk}k\cong\text{Ram}^R(A/R)\otimes_Ak$. Thus, $A/R$ is unramified at $k$ if and only if $\text{Ram}^R(A/R)\otimes_Ak\cong 0$.
\end{proposition}

\begin{proof}
$I_{A/R}^k\otimes_{k\otimes_Rk}k=\text{THH}^R(k;I_{A/R}^k)$ by definition, which is equivalent to $\text{Ram}^R(A/R)\otimes_Ak$ by Remark \ref{RmkBlah2}.
\end{proof}

\begin{remark}
If $A/R$ is totally ramified at $k$, then there is only one $k$-module structure on $I_{A/R}^k$ which induces the intrinsic $k\otimes_Rk$-module structure; there is no ambiguity in speaking of `the' $k$-module $I_{A/R}^k$.

To see this, assume $M$ is a $k\otimes_Rk$-module, and restrict along either of the inclusions $k\to k\otimes_Rk$ to induce a $k$-module $M^\prime$ (with $M^\prime\cong M$ as spectra). If $M$ is induced by a $k$-module $M^{\prime\prime}$, then since the composite $k\to k\otimes_Rk\to k$ is the identity, $M^{\prime\prime}\cong M^\prime$.

This also means that if $f:M\to N$ is a map of $k\otimes_Rk$-modules, and $M,N$ are induced by $k$-modules, then $f$ admits the structure of a $k$-module map (again, by restricting along either inclusion $k\to k\otimes_Rk$).
\end{remark}

\begin{example}\label{ExEtale}
If $A/R$ is \'etale, $A/R$ is unramified for all $k$ (Example \ref{ExED}).
\end{example}

\noindent Some intuition for the general theory can be gleaned from the case $k=A$:

\begin{remark}
Choosing $k=A$, the following are equivalent:
\begin{enumerate}
\item $A/R$ is unramified at $A$;
\item $\text{Ram}^R(A/R)\cong 0$;
\item $A\to\text{THH}^R(A)$ is an equivalence.
\end{enumerate}
\end{remark}

\begin{proposition}\label{PropExample}
Choosing $k=A$, the following are equivalent:
\begin{enumerate}
\item $A/R$ is totally ramified at $A$;
\item $A\otimes_RA$ is an induced $A$-module;
\item $A\otimes_RA\to A$ is an equivalence ($A$ is a \emph{solid} $R$-algebra);
\item $A/R$ is simultaneously totally ramified and unramified at $A$.
\end{enumerate}
\end{proposition}

\begin{proof}
By definition, $A/R$ is totally ramified at $A$ if and only if $I_{A/R}$ is an induced $A$-module. By the fiber sequence $I\to A\otimes_RA\to A$, $I$ is an induced $A$-module if and only if $A\otimes_RA$ is an induced $A$-module. So (1) and (2) are equivalent.

If $A\otimes_RA$ is induced, $A\otimes_RA\xrightarrow{m}A\xrightarrow{i}A\otimes_RA$ is equivalent to the identity, where $i$ is the unique $A$-module map which sends the unit of $A$ to the unit of $A\otimes_RA$. In particular, inclusion $A\to A\otimes_RA$ into either of the $A$ factors is equivalent to $i$, so that $A\xrightarrow{i}A\otimes_RA\xrightarrow{m}A$ is also equivalent to the identity. That is, $m$ and $i$ are inverse, so $m$ is an equivalence. This proves (2) implies (3).

On the other hand, (3) implies that $I_{A/R}^A\cong 0$, so $A/R$ is both totally ramified and unramified at $A$. This proves (3) implies (4), and (4) implies (1) tautologically.
\end{proof}

\noindent More generally, we more-or-less fully understand the extensions which are totally ramified and unramified at the same time.

\begin{proposition}\label{PropBoth}
For arbitrary $k$, the following are equivalent:
\begin{itemize}
\item $A/R$ is both totally ramified and unramified at $k$;
\item $I_{A/R}^k\cong 0$;
\item $k\otimes_Rk\to k\otimes_Ak$ is an equivalence.
\end{itemize}
\end{proposition}

\begin{proof}
By definition, $A/R$ is unramified at $k$ if and only if $0\cong\text{THH}^R(k;I_{A/R}^k)$. If $A/R$ is also totally ramified, then $\text{THH}^R(k;I_{A/R}^k)\cong\text{THH}^R(k)\otimes_kI_{A/R}^k$. Now a copy of $k$ splits off of $\text{THH}^R(k)$ (since $k\to k\otimes_{k\otimes_Rk}k\to k$ is the identity), so a copy of $I^k_{A/R}$ splits off of $\text{THH}^R(k)\otimes_kI_{A/R}^k$. Therefore, the latter is $0$ if and only if $I^k_{A/R}\cong 0$, which by definition means $k\otimes_Rk\to k\otimes_Ak$ is an equivalence.
\end{proof}

\section{Ramification for rings of integers}
\begin{theorem}\label{ThmEquiv}
Suppose $L/K$ is an extension of number fields with rings of integers $A/R$, $\mathfrak{p}$ is a prime in $R$, and $\mathfrak{p}^\prime|\mathfrak{p}$ is a prime in $A$. Then:
\begin{itemize}
\item $A/R$ is totally ramified at $\mathfrak{p}$ (in the classical sense) if and only if $A/R$ is totally ramified at $A/\mathfrak{p}^\prime$ (in the sense of Definition \ref{DefTotalRam});
\item $A/R$ is unramified at $\mathfrak{p}$ (in the classical sense) if and only if $A/R$ is unramified at $A/\mathfrak{p}^\prime$ (in the sense of Definition \ref{DefTotalRam}).
\end{itemize}
\end{theorem}

\noindent The proof boils down to Theorem \ref{MainLemma}, but first we need two lemmas.

\begin{lemma}\label{LemP}
Let $k$ be $p$-complete for some prime $p$. Then:
\begin{itemize}
\item $A/R$ is unramified at $k$ if and only if $A^\wedge_p/R^\wedge_p$ is unramified at $k$;
\item $A/R$ is totally ramified at $k$ if and only if $A^\wedge_p/R^\wedge_p$ is.
\end{itemize}
\end{lemma}

\begin{proof}
By Proposition \ref{PropComplete}, $I_{A^\wedge_p/R^\wedge_p}^k\cong(I_{A/R}^k)^\wedge_p\cong I_{A/R}^k$.
\end{proof}

\begin{lemma}\label{LemTriple}
If $A\to B\to C\to k$ are commutative ring spectra, then any two of the following imply the third:
\begin{itemize}
\item $B/A$ is totally ramified at $k$;
\item $C/B$ is totally ramified at $k$;
\item $C/A$ is totally ramified at $k$.
\end{itemize}
The same is true for unramified extensions.
\end{lemma}

\begin{proof}
By Proposition \ref{PropTriple}, there is a fiber sequence $I_{B/A}^k\to I_{C/A}^k\to I_{C/B}^k$. If any two of these are induced $k$-modules, then so is the third. This proves the totally ramified statement.

Applying $\text{THH}^S(k;-)$ proves the unramified statement.
\end{proof}

\begin{proof}[Proof of Theorem \ref{ThmEquiv}]
By $p$-completing $R$ and $A$ (using Lemma \ref{LemP}), it suffices to prove:
\end{proof}

\begin{proposition}
Suppose $A/R$ is an extension of complete DVRs in mixed characteristic with perfect residue field, with uniformizers $p\in R$ and $\pi\in A$. Then:
\begin{itemize}
\item $A/R$ is totally ramified at $p$ if and only if it is totally ramified at $A/\pi$;
\item $A/R$ is unramified at $p$ if and only if it is unramified at $A/\pi$.
\end{itemize}
\end{proposition}

\begin{proof}
If $A/R$ is totally ramified at $p$, then $A/R$ is totally ramified at $A/\pi$ by Theorem \ref{MainLemma}. If $A/R$ is unramified at $p$, then it is \'etale, so $A/R$ is unramified at $A/\pi$ by Example \ref{ExEtale}.

We still need to prove the converses. By Proposition \ref{PropBoth}, if $A/R$ is both totally ramified and unramified at $A/\pi$, then $0\cong I_{A/R}^k\cong k\otimes_AI_{A/R}\otimes_Ak$, which (together with the fact that $I_{A/R}^k$ is rationally $0$) means $I_{A/R}\cong 0$. Therefore $A\otimes_RA\to A$ is an equivalence, which can only be true if $A=R$.

Now any extension $A/R$ factors $A/R^\prime/R$ where $R^\prime/R$ is unramified and $A/R^\prime$ is totally ramified. If $A/R$ is totally ramified in the sense of Definition \ref{DefTotalRam}, then $R^\prime/R$ is both unramified and (by Lemma \ref{LemTriple}) totally ramified at $A/\pi$, which means $R^\prime=R$ (by the last paragraph) and therefore $A/R$ is totally ramified in the classical sense.

For the same reason, if $A/R$ is unramified in the sense of Definition \ref{DefTotalRam}, then it is unramified in the classical sense.
\end{proof}

\noindent Theorem \ref{ThmEquiv} tells us about the ramification of an extension of rings of integers at each residue field. We could also ask about the ramification at the field of fractions:

\begin{proposition}\label{PropFF}
If $L/K$ is an extension of number fields with rings of integers $A/R$, then $A/R$ is unramified at $L$. In fact, for any $S$, $$\text{Ram}^S(A/R)\otimes_AL\cong 0.$$
\end{proposition}

\begin{proof}
We have $K=R\otimes_\mathbb{Z}\mathbb{Q}$, and so $K\otimes_RK\cong R\otimes_\mathbb{Z}\mathbb{Q}\otimes_\mathbb{Z}\mathbb{Q}\cong K$. Therefore, $\text{THH}^S(R)\otimes_RK\cong\text{THH}^S(K;K\otimes_RK)\cong\text{THH}^S(K)$ and for the same reason $\text{THH}^S(A)\otimes_AL\cong\text{THH}^S(L)$.

By definition, $\text{Ram}^S(A/R)\otimes_AL$ is the fiber of $$\text{THH}^S(R)\otimes_RL\to\text{THH}^S(A)\otimes_AL\cong\text{THH}^S(L).$$ Using $\text{THH}^S(R)\otimes_RL\cong\text{THH}^S(R)\otimes_RK\otimes_KL\cong\text{THH}^S(K)\otimes_KL$, we conclude $\text{Ram}^S(A/R)\otimes_AL$ is the fiber of $\text{THH}^S(K)\otimes_KL\to\text{THH}^S(L)$, or $$\text{Ram}^S(A/R)\otimes_AL\cong\text{Ram}^S(L/K).$$ Since $L/K$ is a Galois extension (therefore \'etale), this is $0$.
\end{proof}

\section{Ramified descent for THH}
\noindent Recall that $\text{Ram}^S(A/R)\otimes_Ak\cong\text{THH}^S(k;I^k_{A/R})$, Remark \ref{RmkBlah2}. If $I^k_{A/R}$ is a $k$-module, then $\text{THH}^S(k;I^k_{A/R})\cong\text{THH}^S(k)\otimes_kI^k_{A/R}$. That is:

\begin{proposition}\label{PropKey}
If $A/R$ is totally ramified at $k$, $$\text{Ram}^S(A/R)\otimes_Ak\cong\text{THH}^S(k)\otimes_kI_{A/R}^k$$ as $\text{THH}^S(R)\otimes_Rk$-modules.
\end{proposition}

\noindent This proposition identifies $\text{Ram}^S(A/R)$ locally at the `point' $\text{Spec}(k)$. If we also have some lifting data, then in some cases we can even determine $\text{Ram}^S(A/R)$ itself:

\begin{definition}
$A/R$ is \emph{totally ramified at $k$ with lifting data} if there is an $A$-module $\Omega^1$ and a map $I_{A/R}\to\Omega^1$ of $A\otimes_RA$-modules, and the composite $$I_{A/R}^k\cong I\otimes_{A\otimes_RA}(k\otimes_Rk)\to\Omega^1\otimes_{A\otimes_RA}(k\otimes_Rk)\to\Omega^1\otimes_Ak$$ is an equivalence. In particular, this implies that $A/R$ is totally ramified at $k$ because $I^k_{A/R}\cong\Omega^1\otimes_Ak$ is a $k$-module.
\end{definition}

\begin{theorem}\label{ThmRam}
Suppose $$\xymatrix{
S\ar[r]\ar[d] &R\ar[d] \\
S^\prime\ar[r] &A
}$$ are commutative ring spectra, and there is a section $S^\prime\to S$ such that the composite $S\to S^\prime\to S$ is equivalent to the identity.

If $A/R$ is totally ramified at $k=A\otimes_{S^\prime}S$ with lifting data, then there is a map $\rho:\text{Ram}^S(A/R)\to\text{THH}^{S^\prime}(A;\Omega^1)$ of $\text{THH}^S(R)\otimes_RA$-modules such that $\rho\otimes_Ak$ is an equivalence.
\end{theorem}

\begin{remark}
Geometrically, $X=\text{Spec}(R)$ and $Y=\text{Spec}(A)$ in $\text{Sch}_{/\text{Spec}(S)}$. The condition in Theorem \ref{ThmRam} is that $Y$ is fibered over some scheme $\text{Spec}(S^\prime)$, and $\text{Spec}(k)$ is the fiber of $Y$ over a point $\text{Spec}(S)\to\text{Spec}(S^\prime)$.

In the example we are interested in (below), $\text{Spec}(S^\prime)$ will be the affine line, pointed at the origin.
\end{remark}

\begin{proof}
$\text{THH}^{S^\prime}(A)\otimes_Ak\cong\text{THH}^{S^\prime}(A)\otimes_{S^\prime}S\cong\text{THH}^S(A\otimes_{S^\prime}S)=\text{THH}^S(k)$. Let $\rho$ be the map $$\text{Ram}^S(A/R)\cong\text{THH}^S(A;I_{A/R})\to\text{THH}^{S^\prime}(A;\Omega^1)\cong\text{THH}^{S^\prime}(A)\otimes_A\Omega^1.$$ Then $\rho\otimes_Ak$ has the form $$\text{Ram}^S(A/R)\otimes_Ak\to\text{THH}^{S^\prime}(A)\otimes_Ak\otimes_A\Omega^1\cong\text{THH}^S(k)\otimes_A\Omega^1,$$ and (unpacking) is the equivalence of Proposition \ref{PropKey}.
\end{proof}

\noindent Now suppose that $A/R$ is a totally ramified extension of $p$-complete DVRs. By Theorem \ref{MainLemma}, $A/R$ is totally ramified at the residue field $k$, with lifting data $I_{A/R}\to\Omega^1_{A/R}$.

Moreover, if $S$ is a commutative ring spectrum, let $S[z]$ be the group ring spectrum $S[\mathbb{N}]=S\otimes\Sigma^\infty_{+}\mathbb{N}$. Regard $A$ as an $S[z]$-algebra by mapping $z$ to a uniformizer. Then $A\otimes_{S_p[z]}S_p\cong k$, so the conditions of Theorem \ref{ThmRam} are satisfied. The theorem produces a map of $\text{THH}^S(R)\otimes_RA$-modules $$\rho:\text{Ram}^S(A/R)\to\text{THH}^{S[z]}(A;\Omega^1_{A/R}).$$

\begin{theorem}\label{MainThm}
If $A/R$ are $p$-completions of rings of integers and $S\to R$ is arbitrary, $\rho$ is an equivalence. That is, there is a fiber/cofiber sequence of $\text{THH}^S(R)\otimes_RA$-modules, $$\text{THH}^{S[z]}(A)\otimes_A\Omega^1_{A/R}\to\text{THH}^S(R)\otimes_RA\to\text{THH}^S(A).$$
\end{theorem}

\begin{proof}
The extension factors $A/R^\prime/R$, where $R^\prime/R$ is unramified and $A/R^\prime$ is totally ramified. Note $\text{Ram}^S(A/R)\cong\text{Ram}^S(A/R^\prime)$ and $\Omega^1_{A/R}\cong\Omega^1_{A/R^\prime}$ by Proposition \ref{PropTriple} and \'etale descent (Example \ref{ExEtale}, since $R^\prime/R$ is \'etale). Therefore, we can assume without loss of generality that $A/R$ is totally ramified.

In this case, $\rho$ is a map of $A$-modules, and $\rho\otimes_Ak$ is an equivalence by Theorem \ref{ThmRam}. If $A$ is the ring of integers in the number field $L$, and $\rho\otimes_AL$ is an equivalence, then $\rho$ must be an equivalence.

It suffices to prove that $\text{Ram}^S(A/R)\otimes_AL\cong\text{THH}^{S[z]}(A;\Omega^1_{A/R})\otimes_AL\cong 0$. The former is Proposition \ref{PropFF}, and the latter is true because $\Omega^1_{A/R}\otimes_AL\cong 0$, since $\Omega^1_{A/R}$ is torsion.
\end{proof}

\noindent In the special cases $S=\mathbb{S}$ or $\mathbb{Z}$, we can calculate explicitly:

\begin{corollary}\label{CorMain}
If $L/K$ is an extension of number fields with rings of integers $A/R$, and $S$ is either $\mathbb{S}$ or $\mathbb{Z}$, then $$\text{Ram}^S_\ast(A/R)\cong\begin{cases}\Omega^1_{A/R}\text{ if }\ast\geq 0\text{ even}, \\ 0\text{ otherwise.}\end{cases}$$
\end{corollary}

\begin{warning}
$\text{Ram}_\ast(A/R)\to\text{Ram}^\mathbb{Z}_\ast(A/R)$ is \emph{not} an isomorphism.
\end{warning}

\begin{proof}
As before, $\text{THH}^{\mathbb{S}_p[z]}(A_p)\otimes_{A_p}k\cong\text{THH}(k)$, where $k$ is the residue field of $A_p$, a perfect field. By \'etale descent on $k/\mathbb{F}_p$ along with the identity $\text{THH}(\mathbb{F}_p)\cong\mathbb{F}_p[\Omega S^3]$ shown in \cite{Blumberg}, $$\text{THH}_\ast^{\mathbb{S}_p[z]}(A_p)\cong H_\ast(\Omega S^3;A_p)\cong A_p[x],$$ where $x$ is a generator in degree $2$. (More details are in \cite{KN} Theorem 3.1.)

For the same reason, but now using $\text{THH}^\mathbb{Z}(\mathbb{F}_p)\cong\mathbb{F}_p[\mathbb{C}P^\infty]$ and $H_\ast(\mathbb{C}P^\infty)$ is a divided power algebra, $$\text{THH}^{\mathbb{Z}_p[z]}_\ast(A_p)\cong\begin{cases}A_p\text{ if }\ast\geq 0\text{ even,} \\ 0\text{ otherwise.}\end{cases}.$$ In either case, by the Universal Coefficient Theorem, $$\text{THH}^{S[z]}_\ast(A;\Omega^1_{A/R})^\wedge_p\cong\text{THH}^{S[z]}_\ast(A)^\wedge_p\otimes_A\Omega^1_{A/R}\cong\begin{cases}(\Omega^1_{A/R})^\wedge_p\text{ if }\ast\geq 0\text{ even,} \\ 0\text{ otherwise,}\end{cases}$$ for every prime $p$, and the result follows by Theorem \ref{MainThm}.
\end{proof}

\begin{example}
If $R$ is a ring of integers, consider the homotopy fiber sequence $\text{Ram}^\mathbb{Z}(R/\mathbb{Z})\to\text{THH}^\mathbb{Z}(\mathbb{Z})\otimes_\mathbb{Z}R\cong R\to\text{THH}^\mathbb{Z}(R)$. By Corollary \ref{CorMain} and the long exact sequence in $\pi_\ast$ associated to the fiber sequence, the ordinary Hochschild homology is $$\text{THH}_\ast^\mathbb{Z}(R)\cong\begin{cases}R\text{ if }\ast=0, \\ \Omega^1_{R/\mathbb{Z}}\text{ if }\ast>0\text{ odd}, \\ 0\text{ otherwise.}\end{cases}$$ However, $\text{THH}_\ast^\mathbb{Z}(R)$ is not hard to calculate using other techniques; see Larsen-Lindenstrauss \cite{LL}.
\end{example}

\end{document}